\documentclass[a4paper, 12pt]{amsart}

\usepackage[T1]{fontenc}

\usepackage[margin=1in]{geometry}

\usepackage{amscd}

\usepackage{amsmath}

\usepackage{amsfonts}

\usepackage{amssymb}

\usepackage{amsthm}

\usepackage{arydshln}

\usepackage{fancyhdr}

\usepackage{verbatim}

\usepackage{tikz-cd}

\usepackage[all]{xy}

\usepackage{color}

\usepackage[colorlinks=true, linkcolor=blue, citecolor=blue,
pagebackref=true]{hyperref}

\usepackage{mathtools} \usepackage{etoolbox}
\patchcmd{\section}{\scshape}{\bfseries}{}{} \makeatletter
\renewcommand{\@secnumfont}{\bfseries} \makeatother

\theoremstyle{definition}

\theoremstyle{plain} \newtheorem{theorem}{Theorem}[section]

\theoremstyle{plain} \newtheorem{lemma}[theorem]{Lemma}

\theoremstyle{plain} 

\theoremstyle{plain} 

\theoremstyle{plain} 

\theoremstyle{remark} \newtheorem*{remark}{Remark}

\theoremstyle{definition} 

\theoremstyle{definition} \newtheorem*{definition*}{Definition}

\theoremstyle{definition} 

\theoremstyle{remark} 

\makeatletter \renewenvironment{proof}[1][\proofname]
{\par\pushQED{\qed}\normalfont\topsep6\p@\@plus6\p@\relax\trivlist\item[\hskip\labelsep\bfseries#1\@addpunct{.}]\ignorespaces}{\popQED\endtrivlist\@endpefalse}
\makeatother

\newcommand{\EE}{\mathbb{E}}

\newcommand{\PP}{\mathbb{P}}

\newcommand{\RR}{\mathbb{R}}

\newcommand{\NN}{\mathbb{N}}

\newcommand{\ZZ}{\mathbb{Z}}

\newcommand{\calE}{\mathcal{E}}

\renewcommand{\leq}{\leqslant} \renewcommand{\geq}{\geqslant}

\DeclarePairedDelimiter{\abs}{\lvert}{\rvert}

\DeclarePairedDelimiter{\set}{\lbrace}{\rbrace}
\DeclarePairedDelimiter{\parens}{\lparen}{\rparen}
\DeclarePairedDelimiter{\brackets}{\lbrack}{\rbrack}

\def\eps{{\varepsilon}}

\def\1int{{[0,1]}}

\title[Approximation by random fractions]{Approximation by random fractions}

\author[L.~Kaziulyt\.e]{Laima Kaziulyt\.e}

\author[F.~A.~Ram{\'i}rez]{Felipe
  A.~Ram{\'i}rez} 


\address{Vilnius University, Lithuania}

\email{laima.kaziulyte@gmail.com}

\address{Wesleyan University,
  Middletown CT, USA} 

\email{framirez@wesleyan.edu}

\date{}

\dedicatory{}

\begin{document}




\maketitle
\thispagestyle{empty}


\begin{abstract}
  We study approximation in the unit interval by rational numbers
  whose numerators are selected randomly with certain
  probabilities. Previous work showed that an analogue of Khintchine's
  Theorem holds in a similar random model and raised the question of
  when the monotonicity assumption can be removed. Informally
  speaking, we show that if the probabilities in our model decay
  sufficiently fast as the denominator increases, then a
  Khintchine-like statement holds without a monotonicity
  assumption. Although our rate of decay of probabilities is unlikely
  to be optimal, it is known that such a result would not hold if the
  probabilities did not decay at all. 
\end{abstract}

\section{Introduction}

Suppose $P=(P_n)_{n=1}^\infty$ is a sequence of subsets
$P_n\subseteq [n]:=\set{1, \dots, n}$. For a function
$\psi:\NN\to [0,1/2]$, let
\begin{equation}\label{W^P}
  W^P(\psi) = \set*{x \in [0,1] : \abs*{x - \frac{a}{n}}<\frac{\psi(n)}{n} \textrm{ for infinitely many } n\in \NN, a\in P_n}. 
\end{equation}
This is the set of real numbers which are ``$\psi$-approximable'' by
fractions whose numerators have been restricted by $P$. One can see
many results of metric Diophantine approximation as statements about
the sets $W^P(\psi)$ for different choices of $P$ and $\psi$. For
example, when $P_n=[n]$ for all $n$, we have Khintchine's Theorem
(\cite{Khintchineonedimensional}, 1924), which states that if $\psi$
is non-increasing, then
\begin{equation}\label{eq:kt}
  \lambda(W^P(\psi)) =
  \begin{cases}
    0 & \textrm{if } \sum_{n=1}^\infty  \psi(n) < \infty \\
    1 & \textrm{if } \sum_{n=1}^\infty \psi(n) = \infty.
  \end{cases}
\end{equation}
Here and throughout, $\lambda$ denotes the Lebesgue measure on
$\RR$. One of the most important problems in metric Diophantine
approximation, the Duffin--Schaeffer conjecture (\cite{DSC}, 1941),
concerns the setting where $P_n= \set{a \in [n] : (a,n)=1}$ for all
$n$, where $(m,n)$ denotes the greatest common divisor of
$m,n\in\NN$. It states that for any $\psi$ we have
\begin{equation}\label{eq:ds}
  \lambda(W^P(\psi)) =
  \begin{cases}
    0 & \textrm{if } \sum_{n=1}^\infty \frac{\varphi(n)\psi(n)}{n}  < \infty \\
    1 & \textrm{if } \sum_{n=1}^\infty \frac{\varphi(n)\psi(n)}{n} =
    \infty.
  \end{cases}
\end{equation}
Importantly, the Duffin--Schaeffer conjecture does not assume that
$\psi$ is monotonic. (D.~Koukoulopoulos and J.~Maynard~\cite{KM} have
recently announced a proof of the Duffin--Schaeffer conjecture.) Many
other results for different choices of $P$ are known. They often
appear under the heading ``Diophantine approximation with restricted
numerators'' (see for example~\cite[Chapter~6]{Harman}).

Recently, the second author studied the sets $W^P(\psi)$ where $P$ is
obtained randomly from the space of all possible such sequences,
endowed with a probability measure. Specifically, for each $n$, choose
$P_n$ uniformly at random from among all $f(n)$-element subsets of
$[n]$, where $f$ is a fixed integer sequence. Here, analogues of
Khintchine's Theorem hold. For example, under an average growth
condition on $f$,~(\ref{eq:kt}) will almost surely hold for any $\psi$
that is monotonic~\cite[Corollary~1.3]{random_fractions}. Furthermore,
the monotonicity assumption cannot be removed if the growth of $f(n)$
satisfies $f(n)/\parens*{n (\log\log n)^{-1}}\to\infty$ as
$n\to\infty$~\cite[Theorem~1.8]{random_fractions}.  A randomized
version of the Duffin--Schaeffer conjecture---where $f$ is the Euler
totient function $\varphi$---is still open. Specifically, with
$f=\varphi$, does~(\ref{eq:ds}) almost surely hold for all $\psi$?  In
general, the question of which sequences $f$ admit the removal of a
monotonicity assumption on $\psi$, is still
open~\cite[Question~1.7]{random_fractions}.

In this note we study a random model for which we are able to
establish a Khintchine-like result without a monotonicity
assumption. Instead of the randomization introduced in
\cite{random_fractions} (discussed above) we randomize
$P=(P_n)_{n=1}^\infty$ in the following way. Fix a sequence
$\set{p_n}_{n=1}^\infty$ of probabilities, $0 \leq p_n \leq 1$, and
for each $a\in [n]$, decide with probability $p_n$ whether $a$ is
included in $P_n$. Our main result is the
following.

\begin{theorem}\label{thm}
  If there exists an $\eps>0$ such that $p_n \ll (\log n)^{-\eps}$, then for any $\psi:\NN \to [0,1/2]$, we
  almost surely have
\begin{equation*}
  \lambda(W^P(\psi)) =
  \begin{cases}
    0 & \textrm{if } \sum_{n=1}^\infty p_n \psi(n)  < \infty \\
    1 & \textrm{if } \sum_{n=1}^\infty p_n \psi(n) = \infty
  \end{cases}
\end{equation*}
(where $P$ is randomized as described just above the theorem).
\end{theorem}

\begin{remark}[On notation]
  For two functions $f,g:\NN\to \RR_{\geq 0}$, we use $f \ll g$ to
  denote that there is a constant $C>0$ such that $f(n) \leq C g(n)$
  for all sufficiently large $n$. Unless otherwise noted, the implied
  constant is absolute.
\end{remark}

\begin{remark}[Comparison with randomization from~\cite{random_fractions}]
  The operative difference between our randomization and the one
  in~\cite{random_fractions} is that, here, the events that different $a$'s from $[n]$ are contained in $P_n$
  are totally independent, whereas in the
  other model they are not. The effect of this independence is
  to make certain calculations (particularly variance bounds)
  easier. Yet the two random models are related in a natural way. In
  the present work, the expected cardinality of $P_n$ is $np_n$, so
  for the sake of comparison one may think of $n p_n$ as playing the
  role of $f(n)$. 
\end{remark}

Theorem~\ref{thm} is reminiscent of some of the results that preceded
the Koukoulopoulos--Maynard proof of the Duffin--Schaeffer
conjecture. Namely,
in~\cite{AistleitneretalExtraDivergence,AistleitnerDS,BHHVextraii,HPVextra}
the Duffin--Schaeffer conjecture was proved under certain ``extra
divergence'' and ``slow divergence'' assumptions on the sum
in~(\ref{eq:ds}). The most recent of these
was~\cite{AistleitneretalExtraDivergence}, where Aistleitner \emph{et
  al.}~proved the conjecture with the assumption that there is some
$\eps>0$ for which $\sum \frac{\varphi(n)\psi(n)}{n (\log n)^\eps}$
diverges. Relatedly, the first author showed that Khintchine's Theorem
holds without the monotonicity assumption if one assumes that there
exists some $\eps>0$ for which $\sum \frac{\psi(n)}{(\log n)^\eps}$
diverges~\cite{LaimaPaper}. The proof of Theorem~\ref{thm} relies in
part on a ``partial reduction'' method from~\cite{LaimaPaper}, to be
explained in due course.

We hasten to point out that all of the results in the previous
paragraph are subsumed by the Duffin--Schaeffer conjecture
itself. However, none of the random results and conjectures seem to
be. In particular, Theorem~\ref{thm} can be viewed as partial progress
toward a randomized version of the Duffin--Schaeffer conjecture,
analogous to the progress toward the classical Duffin--Schaeffer
conjecture that the results in the previous paragraph
represented. With our randomization, the analogue of the
Duffin--Schaeffer conjecture would come from setting
$p_n = \varphi(n)/n$ for each $n$ in the statement of
Theorem~\ref{thm}. With this analogy in mind, it is believable that a
statement like Theorem~\ref{thm}, with $p_n \ll (\log n)^{-\eps}$
replaced by $p_n \ll (\log\log n)^{-\eps}$ (now with $\eps\geq 1$), is
also true, and is perhaps a natural ``next step'' in our present
efforts (short of pursuing the full analogue of
Duffin--Schaeffer). This would parallel a recent ``extra divergence''
breakthrough of Aistleitner (appearing on the
arXiv~\cite{AistleitnerRecent} but not intended for publication
because it was subsumed by~\cite{KM}), where the extra divergence
assumption from~\cite{AistleitneretalExtraDivergence} is relaxed to
there existing $\eps>0$ for which
$\sum \frac{\varphi(n)\psi(n)}{n(\log \log n)^\eps}$ diverges.

\section{Strategy of the proof}
\label{sec:strategy-proof}

Let us hereby fix $\eps\in (0,1)$ and $p_n$ ($n\in\NN$) as in the
statement of Theorem~\ref{thm}. (It is easy to see that no generality
is lost in making this stronger assumption on $\eps$.) Now let $P$ be
chosen as in the theorem. Note that for any $\psi: \NN \to [0, 1/2]$
we can write
\begin{equation*}
W^P(\psi) = \limsup_{n\to\infty} \calE_n^P,
\end{equation*}
where
\begin{equation*}
  \calE_n^P:=\calE_n^P(\psi) = \bigcup_{a\in
    P_n}\parens*{\frac{a-\psi(n)}{n}, \frac{a+\psi(n)}{n}}
\end{equation*}
for $\psi:\NN\to [0,1/2]$ and $n\in \NN$.

The convergence parts of the results in the previous section are all
easy applications of the Borel--Cantelli lemma to the ``limsup'' set
$W^P(\psi)$. The divergence parts require more work. One common
strategy is to first establish some form of independence among the
sets $\calE_n^P$ and then use the following lemma to show that $\lambda\parens*{W^P(\psi)}$ is
positive. (Getting full measure usually requires another ingredient,
to be discussed.) 

\begin{lemma}[{\cite[Lemma~2.3]{Harman}}]\label{lem:reversebc}
  Let $(X,\lambda)$ be a finite measure space, and $\calE_n$ a
  sequence of measurable subsets of $X$ such that
  $\sum\lambda(\calE_n)$ diverges. Then 
  \begin{equation}\label{eq:qia}
    \lambda\parens*{\limsup_{n\to\infty} \calE_n} \geq \limsup_{N\to\infty} \parens*{\sum_{n=1}^N\lambda(\calE_n)}^2\parens*{\sum_{m,n = 1}^N\lambda(\calE_m\cap\calE_n)}^{-1}. 
  \end{equation}
\end{lemma}

Rather than apply this lemma to the sets $\calE_n^P$, we work with
certain subsets of $\calE_n^P$, which we presently describe. Instead
of considering all of $P$, we restrict attention to the elements
$a\in P_n$ such that $(a,n)\le(\log n)^{\varepsilon/2}$.  (This is the
``partial reduction'' which we shortly mentioned in the previous section.) Set the
notation
\begin{align}
S_n &= \set*{a \in [n] : (a,n) \leq (\log n)^{\eps/2}}\label{S_n}\\
\intertext{and}
  \calE_n^{S}:=\calE_n^{S}(\psi) &= \bigcup_{a\in S_n}\parens*{\frac{a-\psi(n)}{n}, \frac{a+\psi(n)}{n}},\label{E_n}
\end{align}
where $S=(S_n)_{n\in\NN}$. The sets $\calE_n^{S}$ were introduced
in~\cite{LaimaPaper}, where they were shown to be of comparable size
and to have better independence properties than the analogous
non-reduced sets with $a\in[n]$. Let $Q=(Q_n)_{n\in\NN}$, where
$Q_n = P_n \cap S_n$, and define $\calE_n^Q$ in the obvious
way. Throughout most of this paper, we will work directly with the
sets $\calE_n^Q$ instead of $\calE_n^P$.

When the right-hand side of~(\ref{eq:qia}) is positive, one says that
the sets $\calE_n$ are quasi-independent on average (QIA), and
Lemma~\ref{lem:reversebc} implies that the corresponding ``limsup'' set
has positive measure. In many situations one can upgrade this to
full measure by applying the appropriate ``zero-one law'' (see for
instance the zero-one laws of Cassels~\cite{cassels} and
Gallagher~\cite{gall}). But no such zero-one law has been proved in
our random setting. To get around this, we seek to apply
Lemma~\ref{lem:reversebc} with $\calE_n = \calE_n^Q\cap J$, where $J$
is some arbitrarily fixed subinterval of $[0,1]$. We will show that
almost surely the right-hand side of~(\ref{eq:qia}) is bounded below
by an absolute constant times the length of $J$. Then the Lebesgue
density theorem together with Lemma~\ref{lem:reversebc} will imply
that $\limsup\calE_n^Q$ has full measure in the unit interval, hence
so does its superset $W^P(\psi)$.

The content of the rest of this paper is as follows. We estimate the expected value and variance of the pairwise overlap $\lambda\parens*{\calE_m^Q\cap\calE_n^Q \cap J}$ (on average)
in Section \ref{sec:overlaps} and \ref{sec:vari-calc} respectively. Section \ref{sec:measures} is for showing that these estimates are small and can be compared with the measures of approximation sets $\lambda\parens*{\calE_n^Q \cap J}$. The application of Chebyshev's  inequality then implies QIA, almost surely, in Section~\ref{sec:proof}. In the same section the proofs of both divergence and convergence part of the theorem are provided. 

\section{Measures of approximation sets}
\label{sec:measures}
The result of Lemma~\ref{lem:S} below is used in
Lemma~\ref{lem:asmoas} , which, in turn, later is needed to prove QIA,
almost surely.

The next lemma, which is adapted from~{\cite[Lemma 2 and Corollary
  3]{LaimaPaper}}, shows that we have neither lost too many elements
nor too much measure by restricting attention to partially reduced
fractions.

\begin{lemma}\label{lem:S}
  For any fixed arc $J\subset \RR/\ZZ$ we have
  \begin{equation*}
    \abs{S_n\cap nJ} \gg \lambda(J)\eps n
    \quad
    \textrm{and hence}
    \quad
    \lambda\parens*{\calE_n^S\cap J} \gg \lambda(J)\eps\psi(n)
  \end{equation*}
  with universal implied constants, where $S_n$ and $\calE_n^S$ are
  defined by (\ref{S_n}) and (\ref{E_n}) respectively for fixed
  $\varepsilon\in(0,1)$.
\end{lemma}

\begin{remark}[On notation]
  We denote by $\abs{S}$ the cardinality of a finite set $S$.
\end{remark}

\begin{proof}
  Let us consider $S_{n}^{(1)}:=\set*{a\in [n] : (a,n)=1}$. Then, from
  {\cite[Lemma 5.2]{random_fractions}} we get that there is some
  $n_0(J)$ such that
  \begin{equation}
    |S_{n}^{(1)} \cap n J|\ge\frac{1}{2}\varphi(n)\lambda(J) \qquad \textrm{for all} \qquad
    n\ge n_0(J).\label{eq:phicount}
  \end{equation}
  Similarly, if $d|n$ denote $S_n^{(d)}:=\set*{a\in [n] : (a,n)=d}$.
  Notice that $S_n^{(d)}$ is exactly the set
  $\set*{d\cdot b : b\in [n/d], (b,n/d)=1}$. That is,
  $S_n^{(d)} = d S_{n/d}^{(1)}$. Notice also that
  \begin{equation*}
    S_n\cap nJ = \bigcup_{\substack{d \mid n \\ d \leq (\log n)^{\eps/2}}} S_n^{(d)} \cap nJ,
  \end{equation*}
  and that the union is disjoint. Therefore, 
  \begin{align*}
    \abs{S_n\cap nJ}  &= \sum_{\substack{d \mid n \\ d \leq (\log n)^{\eps/2}}} \abs*{S_n^{(d)} \cap nJ} \\
               &= \sum_{\substack{d \mid n \\ d \leq (\log n)^{\eps/2}}} \abs*{S_{n/d}^{(1)} \cap \parens*{\frac{n}{d}} J}.
  \end{align*}
  From~(\ref{eq:phicount}), each summand has a lower bound
  $$
  \abs*{S_{n/d}^{(1)}\cap \parens*{\frac{n}{d}}J} \ge
  \frac{1}{2}\varphi\left(\frac{n}{d}\right)\lambda(J)
  $$
  as long as $n/d \ge n_0(J)$. In particular, if
  $n (\log n)^{-\eps/2} \geq n_0(J)$, then
  \begin{align*}
    \abs{S_n\cap nJ} &\geq \frac{1}{2} \lambda(J) \underbrace{\sum_{\substack{d \mid n \\ d \leq (\log n)^{\eps/2}}} \varphi\parens*{\frac{n}{d}}}_{\gg \eps n}\\
              &\gg \lambda(J)\eps n,
  \end{align*}
  where the last step is achieved in~\cite[Lemma 2]{LaimaPaper}. This
  proves the lemma.
\end{proof}

In the following lemma we show that the sum of interest in the
statement of Theorem~\ref{thm} is almost surely a lower bound for the
sum of measures of approximation sets for infinitely many partial sums.

\begin{lemma}[Almost sure measure of approximation sets]\label{lem:asmoas}
  Assume we are in the divergence case of Theorem~\ref{thm}. Then
  there exists an absolute constant $C>0$ and an increasing integer
  sequence $\set{N_t}_{t=1}^\infty$ such that for any fixed arc
  $J\subset\RR/\ZZ$ it is almost surely the case that
  \begin{equation*}
    \sum_{n=1}^{N_t} \lambda\parens*{\calE_n^Q\cap J} \geq C \lambda(J)\varepsilon \sum_{n=1}^{N_t} p_n\psi(n)
  \end{equation*}
  holds for all but finitely many $t$. 
\end{lemma}

\begin{proof}
  Note that
  \begin{align*}
   \lambda\parens*{\calE_n^Q\cap J} \geq \frac{\psi(n)}{n} \abs{Q_n \cap nJ}.
  \end{align*}
  Recall that any element of $S_n$ has a probability $p_n$ of being
  selected for a membership in $Q_n$. Therefore,
  $\EE\parens*{\abs*{Q_n\cap nJ}} = p_n \abs{S_n\cap nJ}$, and
  \begin{align*}
   \EE\parens*{\lambda\parens*{\calE_n^Q\cap J}}\geq \frac{\psi(n)}{n} p_n\abs{S_n\cap nJ}.
  \end{align*}
  Hence, 
  \begin{align}
    \EE\parens*{\sum_{n=1}^{N} \lambda\parens*{\calE_n^Q\cap J}} 
    &\geq \sum_{n=1}^{N} \frac{\psi(n)}{n} p_n\abs{S_n\cap nJ} \notag \\
    &\gg \lambda(J)\eps \sum_{n=1}^{N} p_n\psi(n)\label{eq:expectedmeasure}
  \end{align}
  by Lemma~\ref{lem:S}.
  On the other hand, for a variance we have
  $\sigma^2\parens*{\abs*{Q_n\cap nJ}} = p_n (1-p_n)\abs{S_n\cap nJ}$,
  hence
  \begin{align*}
    \sigma^2\parens*{\lambda\parens*{\calE_n^Q\cap J}} &\leq \parens*{2\frac{\psi(n)}{n}}^2 p_n(1-p_n)\abs{S_n},
  \end{align*}
  since 
  $$\lambda\parens*{\calE_n^Q\cap J} \le \frac{2\psi(n)}{n} \abs{Q_n \cap nJ}.$$
  Now, since all selections for $Q_n, n=1,2,\ldots$ occur independently, we have
  \begin{align*}
    \sigma^2\parens*{\sum_{n=1}^{N} \lambda\parens*{\calE_n^Q\cap J}} 
    &\leq \sum_{n=1}^{N} \parens*{2\frac{\psi(n)}{n}}^2 p_n(1-p_n)\abs{S_n}\\
    &\leq\sum_{n=1}^{N} 4 \frac{\psi(n)^2}{n} p_n(1-p_n) \\
    &\leq \sum_{n=1}^{N} 4 \frac{\psi(n)}{n} p_n \\
    &\overset{\textrm{(\ref{eq:expectedmeasure})}}{\ll}_{\varepsilon, J} \EE\parens*{\sum_{n=1}^{N} \lambda\parens*{\calE_n^Q\cap J}}.
  \end{align*}
  Thus Chebyshev's inequality implies
  \begin{align*}
    \PP\brackets*{\sum_{n=1}^{N} \lambda\parens*{\calE_n^Q\cap J} \leq \frac{1}{2}\EE\parens*{\sum_{n=1}^{N} \lambda\parens*{\calE_n^Q\cap J}}} 
    &\ll_{\varepsilon, J} \frac{1}{\EE\parens*{\sum_{n=1}^{N} \lambda\parens*{\calE_n^Q\cap J}}}\\
    &\overset{\textrm{(\ref{eq:expectedmeasure})}}{\ll}_{\varepsilon, J} \parens*{\lambda(J)\eps \sum_{n=1}^{N} p_n\psi(n)}^{-1}.
  \end{align*}
  Let $\set{N_t}_{t=1}^\infty$ be an increasing integer sequence which is sparse
  enough that
  \begin{equation*}
    \sum_{t=1}^\infty \parens*{\sum_{n=1}^{N_t} p_n\psi(n)}^{-1} < \infty. 
  \end{equation*}
  Then, by the Borel--Cantelli lemma, we almost surely have
  \begin{equation*}
    \sum_{n=1}^{N_t}\lambda\parens*{\calE_n^Q\cap J} \geq \frac{1}{2}\EE\parens*{\sum_{n=1}^{N_t} \lambda\parens*{\calE_n^Q\cap J}}
  \end{equation*}
  for all but finitely many $t$. Combining this
  with~(\ref{eq:expectedmeasure}) proves the
  lemma. 
\end{proof}




\section{Expected overlaps of approximation sets}
\label{sec:overlaps}

It is worth noting that the results from Section~\ref{sec:measures}
hold for any subinterval $J\subset [0,1]$. By contrast, the
results in this and the next section are stated with a randomized
interval $J$. That is, we randomly choose $P$ as before, and also
independently randomly choose $J$, so that our probability space is
extended. To avoid confusion, we will use the notation $\overline\PP$,
$\overline\EE$ and $\overline\sigma^2$ to denote the probability measure,
expectation and variance, respectively, in the extended space. We use
$\PP, \EE, \sigma^2$ for the space of $P$'s.

\begin{lemma}[Expected average overlap estimates]\label{lem:expectedoverlaps}
  Let $M\in \NN$ and express $[0,1] = J_1\cup\dots\cup J_M$, where the
  $J_i$ are disjoint subintervals of length $1/M$. Let
  $J\in \set{J_1, \dots, J_M}$ be chosen randomly and uniformly. Then
  \begin{equation*}
    \overline\EE\brackets*{\sum_{1\leq m, n \leq N} \lambda\parens*{\calE_m^Q\cap\calE_n^Q \cap J}} \ll \lambda(J) \parens*{\sum_{n=1}^N p_n \psi(n)}^2.
  \end{equation*}
\end{lemma}

\begin{proof}
  First, observe that
  \begin{align}
    \overline\EE\brackets*{\sum_{1\leq m, n \leq N} \lambda\parens*{\calE_m^Q\cap\calE_n^Q \cap J}} 
    &= \EE\brackets*{\sum_{1\leq m, n \leq N} \frac{1}{M} \sum_{i=1}^M \lambda\parens*{\calE_m^Q\cap\calE_n^Q \cap J_i}}\notag \\
    &= \lambda(J) \EE\brackets*{\sum_{1\leq m, n \leq N} \lambda\parens*{\calE_m^Q\cap\calE_n^Q}},\label{eq:inviewof}
  \end{align}
  so let us bound this last expectation.
  
  Note that for any $m\neq n$, $\calE_m^S$ and $\calE_n^S$ are unions
  of disjoint intervals, hence so is their intersection
  $\calE_m^S\cap \calE_n^S$. Since $\calE_m^Q$ and $\calE_n^Q$ are
  obtained by selecting from the constituent intervals of $\calE_m^S$
  and $\calE_n^S$, it is clear that the intersection
  $\calE_m^Q\cap \calE_n^Q$ is a union of disjoint intervals, each
  chosen from $\calE_m^S\cap \calE_n^S$. Indeed, any of the intervals
  making up $\calE_m^S\cap \calE_n^S$ has probability $p_mp_n$ of
  being selected. Therefore, by linearity of expectation, we have
  \begin{equation*}
    \EE\brackets*{\lambda\parens*{\calE_m^Q\cap\calE_n^Q}} = p_mp_n\lambda\parens*{\calE_m^S\cap\calE_n^S}.
  \end{equation*}
  On the other hand, if $m=n$, then we have 
  \begin{equation*}
    \EE\brackets*{\lambda\parens*{\calE_m^Q\cap\calE_n^Q}} = p_n\lambda\parens*{\calE_n^S}.
  \end{equation*}
  Thus
  \begin{align}
    \EE\brackets*{\sum_{m,n=1}^N \lambda\parens*{\calE_m^Q\cap\calE_n^Q}}\notag
    &=\sum_{m,n=1}^N\EE\brackets*{\lambda\parens*{\calE_m^Q\cap\calE_n^Q}} \\
    &=\sum_{\substack{m,n=1 \\ m\neq n}}^Np_mp_n\lambda\parens*{\calE_m^S\cap\calE_n^S} + \sum_{n=1}^Np_n\lambda\parens*{\calE_n^S} \notag \\
    &\le 8\sum_{m=1}^N\sum_{n=1}^Np_mp_n\psi(m)\psi(n)+\sum_{n=1}^N p_n\lambda\parens*{\calE_n^S}\notag\\
    &\quad+4\sum_{n=1}^Np_n\frac{\psi(n)}{n}\sum_{\substack{n>m\ge n/(\log n)^{\eps/2}\\ (n,m)\ge n/(\log n)^{\eps/2}}}p_m\cdot m. \label{eq:centerscoincide}
  \end{align}
  The first of the three summand after the last inequality above  comprises
  the elementary bounds on measures of intervals whose centers do not
  coincide (see~\cite[Page 39]{Harman}); the last summand accounts for
  the intervals whose centers do coincide (see~\cite[Page 176]{Harman}
  and~\cite[Page 7]{LaimaPaper}). Remembering that
  $p_m \ll (\log m)^{-\eps}$,

  \begin{equation*}
    \sum_{\substack{n>m\ge n/(\log n)^{\eps/2}\\ (n,m)\ge n/(\log n)^{\eps/2}}}p_m\cdot m 
 \ll \frac{n}{(\log n)^\eps}\sum_{\substack{d\mid n \\ d\leq n/(\log n)^{\eps/2}}} d \ll n.
  \end{equation*} 
  Putting this into~(\ref{eq:centerscoincide}), we have
  \begin{align*}
    \EE\brackets*{\sum_{m,n=1}^N \lambda\parens*{\calE_m^Q\cap\calE_n^Q}}
    &\ll \sum_{m=1}^N\sum_{n=1}^Np_mp_n\psi(m)\psi(n)+\sum_{n=1}^Np_n\psi(n) \\
    &\ll \parens*{\sum_{n=1}^Np_n\psi(n)}^2.
  \end{align*}
  The lemma is proved by combining this with~(\ref{eq:inviewof}).
\end{proof}

\section{Variance calculations}
\label{sec:vari-calc}

As in Section~\ref{sec:overlaps}, we continue to view the subinterval
$J\subset[0,1]$ as a random object. 

\begin{lemma}\label{lem:varone}
  Let $J\subset [0,1]$ be randomly uniformly chosen from
  $\set{J_1, \dots, J_M}$ for a fixed $M\in\NN$ as in Lemma~\ref{lem:expectedoverlaps}. Then
  \begin{equation}
    \overline\sigma^2\brackets*{\sum_{1\leq m,n\leq N} \lambda\parens*{\calE_m^Q\cap\calE_n^Q\cap J}} 
    \leq \sum_{1\leq m,n\leq N} \overline\sigma^2\brackets*{\lambda\parens*{\calE_m^Q\cap\calE_n^Q\cap J}}.\label{var_of_sums<=sum_of_var}
  \end{equation}
\end{lemma}

\begin{proof}
  Beginning with a definition, we have
  \begin{align}
    &\overline\sigma^2\brackets*{\sum_{1\leq m,n\leq N}\lambda\parens*{\calE_m^Q\cap\calE_n^Q\cap J}} \notag\\
    &=\overline\EE\brackets*{\parens*{\sum_{1\leq m,n\leq N}\lambda\parens*{\calE_m^Q\cap\calE_n^Q\cap J}}^2} - \parens*{\overline\EE\brackets*{\sum_{1\leq m,n\leq N}\lambda\parens*{\calE_m^Q\cap\calE_n^Q\cap J}}}^2 \nonumber \\
    &=\sum_{1\leq k,\ell, m,n\leq N}\overline\EE\brackets*{\lambda\parens*{\calE_k^Q\cap\calE_\ell^Q\cap J}\lambda\parens*{\calE_m^Q\cap\calE_n^Q\cap J}} \label{eq:Ell}\\
    &\quad- \sum_{1\leq k,\ell, m,n\leq N}\overline\EE\brackets*{\lambda\parens*{\calE_k^Q\cap\calE_\ell^Q\cap J}}\EE\brackets*{\lambda\parens*{\calE_m^Q\cap\calE_n^Q\cap J}}.\label{eq:ElEl}
  \end{align}
  The above summations take place over all possible
  $1\leq k, \ell, m, n\leq N$. Let us split the task up over the
  following indexing subsets:
  \begin{align*}
    A &= \set{1\leq k, \ell, m, n\leq N : \set{k, \ell}\cap\set{m, n}=\emptyset}\\
    B &= \set{1\leq k, \ell, m, n\leq N : \set{k, \ell}\cap\set{m, n}=\set{k}, k\neq l}\\
    C &= \set{1\leq k, \ell, m, n\leq N : \set{k, \ell}\cap\set{m, n}=\set{k,\ell}},
  \end{align*}
  where, of course, the intention is to consider these with the
  appropriate permutations as well.
  The advantage is that the sums over $A$ and $B$ are relatively easy
  to handle. Indeed, notice that whenever $(k,\ell, m, n)$ belongs to
  $A$, the random variables
  $\lambda\parens*{\calE_k^Q\cap\calE_\ell^Q\cap J}$ and
  $\lambda\parens*{\calE_m^Q\cap\calE_n^Q\cap J}$ are independent, hence the
  expectation of their product is the product of their
  expectations. That is, the sum over $A$ in~(\ref{eq:Ell}) cancels
  the sum over $A$ in~(\ref{eq:ElEl}). Now consider the sums in~(\ref{eq:Ell}) and~(\ref{eq:ElEl}) over $B$.
  We claim that for any $(k,\ell, k, n)$ such that
  $\ell\neq n$, we have
  \begin{equation*}
\overline\EE\brackets*{\lambda\parens*{\calE_k^Q\cap\calE_\ell^Q\cap J}\lambda\parens*{\calE_k^Q\cap\calE_n^Q\cap J}} -\overline\EE\brackets*{\lambda\parens*{\calE_k^Q\cap\calE_\ell^Q\cap J}}\overline\EE\brackets*{\lambda\parens*{\calE_k^Q\cap\calE_n^Q\cap J}} \leq 0. 
  \end{equation*}
  To see that, we partition a phase space according to the different
  states $\calE_k^Q\cap J$ can occupy, of which there are finitely many. Let
  us label those states $\set{s_1, \dots, s_d}$. Notice that given any
  of these states, the random variables
  $\lambda\parens*{\calE_k^Q\cap\calE_\ell^Q\cap J}$ and
  $\lambda\parens*{\calE_k^Q\cap\calE_n^Q\cap J}$ are (conditionally)
  independent. Therefore, 
\begin{align*}
  &\overline\EE\brackets*{\lambda\parens*{\calE_k^Q\cap\calE_\ell^Q\cap J}\lambda\parens*{\calE_k^Q\cap\calE_n^Q\cap J}} -\overline\EE\brackets*{\lambda\parens*{\calE_k^Q\cap\calE_\ell^Q\cap J}}\overline\EE\brackets*{\lambda\parens*{\calE_k^Q\cap\calE_n^Q\cap J}} \\
  &= \sum_{i=1}^d \overline\EE\brackets*{\lambda\parens*{\calE_k^Q\cap\calE_\ell^Q\cap J}\lambda\parens*{\calE_k^Q\cap\calE_n^Q\cap J}\mid s_i}\overline\PP(s_i) -\overline\EE\brackets*{\lambda\parens*{\calE_k^Q\cap\calE_\ell^Q\cap J}}\overline\EE\brackets*{\lambda\parens*{\calE_k^Q\cap\calE_n^Q\cap J}} \\
  &\overset{\substack{\textrm{Cond.} \\ \textrm{indep.}}}{=} \sum_{i=1}^d \overline\EE\brackets*{\lambda\parens*{\calE_k^Q\cap\calE_\ell^Q\cap J}\mid s_i}\overline\EE\brackets*{\lambda\parens*{\calE_k^Q\cap\calE_n^Q\cap J}\mid s_i}\overline\PP(s_i) \\
  &\quad- \parens*{\sum_{i=1}^d\overline\EE\brackets*{\lambda\parens*{\calE_k^Q\cap\calE_\ell^Q\cap J}\mid s_i}\overline\PP(s_i)}\parens*{\sum_{i=1}^d\overline\EE\brackets*{\lambda\parens*{\calE_k^Q\cap\calE_n^Q\cap J}\mid s_i}\overline\PP(s_i)} \\
  &\overset{\substack{\textrm{Reverse}\\ \textrm{H{\"o}lder}}}{\leq} \sum_{i=1}^d \overline\EE\brackets*{\lambda\parens*{\calE_k^Q\cap\calE_\ell^Q\cap J}\mid s_i}\overline\EE\brackets*{\lambda\parens*{\calE_k^Q\cap\calE_n^Q\cap J}\mid s_i}\overline\PP(s_i) \\
  &\quad- \parens*{\sum_{i=1}^d\parens*{\overline\EE\brackets*{\lambda\parens*{\calE_k^Q\cap\calE_\ell^Q\cap J}\mid s_i}}^2\overline\PP(s_i)}^{1/2}\parens*{\sum_{i=1}^d\parens*{\overline\EE\brackets*{\lambda\parens*{\calE_k^Q\cap\calE_n^Q\cap J}\mid s_i}}^2\overline\PP(s_i)}^{1/2}\\
  &\overset{\substack{\textrm{Cauchy--}\\ \textrm{Schwarz}}}{\leq} 0.
  \end{align*}
  The lemma is proved, noticing that the sums in~(\ref{eq:Ell}) and~(\ref{eq:ElEl}) over $C$ give exactly the right-hand side of~(\ref{var_of_sums<=sum_of_var}). 
\end{proof}

\begin{lemma}\label{lem:vartwo}
  Let $M\in \NN$ and let $J\in \set{J_1, \dots, J_M}$ be chosen
  randomly uniformly as in Lemma~\ref{lem:expectedoverlaps}. Then,
  for $m\neq n$,
  \begin{equation*}
    \overline\sigma^2\brackets*{\lambda\parens*{\calE_m^Q\cap \calE_n^Q\cap J}} \ll p_mp_n\psi(m)\psi(n) + \EE\brackets*{\lambda\parens*{\calE_m^Q\cap\calE_n^Q}}.
  \end{equation*}
\end{lemma}

\begin{proof}
  Let $m < n$. We have
  \begin{align}
    \overline\sigma^2\brackets*{\lambda\parens*{\calE_m^Q\cap \calE_n^Q\cap J}} 
    &= \frac{1}{M}\sum_{i=1}^M \EE\brackets*{\lambda\parens*{\calE_m^Q\cap \calE_n^Q\cap J_i}^2} - \frac{1}{M^2}\parens*{\sum_{i=1}^M \EE\brackets*{\lambda\parens*{\calE_m^Q\cap \calE_n^Q\cap J_i}}}^2 \notag \\
    &\leq \frac{1}{M}\sum_{i=1}^M \EE\brackets*{\lambda\parens*{\calE_m^Q\cap \calE_n^Q\cap J_i}^2} \notag \\
    &= \frac{1}{M}\sum_{i=1}^M \sigma^2\brackets*{\lambda\parens*{\calE_m^Q\cap \calE_n^Q\cap J_i}}+ \frac{1}{M} \sum_{i=1}^M \EE\brackets*{\lambda\parens*{\calE_m^Q\cap \calE_n^Q\cap J_i}}^2 \notag \\
    &\leq \frac{1}{M}\sum_{i=1}^M \sigma^2\brackets*{\lambda\parens*{\calE_m^Q\cap \calE_n^Q\cap J_i}}+ \EE\brackets*{\lambda\parens*{\calE_m^Q\cap \calE_n^Q}}. \label{eq:twoterms}
  \end{align}
  Now we only have to treat the first term in the last inequality above.
 
  We lose no generality in assuming $\psi$ takes values in
  $[0,1/4]$. The effect of this extra assumption is to ensure that no
  interval from $\calE_n^S$ can intersect more than one interval from
  $\calE_m^S$. To see it, note that any two intervals of $\calE_m^S$
  are separated by a distance of at least $1/(2m)$. Meanwhile, the
  intervals of $\calE_n^S$ have length $2\psi(n)/n < 1/(2n) < 1/(2m)$.

  Now, even though the intervals making up $\calE_m^S\cap\calE_n^S\cap J_i$
  are not chosen independently to form $\calE_m^Q\cap\calE_n^Q\cap J_i$, they
  can be grouped into  finitely many independent sub-unions, each
  corresponding to a different interval from $\calE_m^S\cap J_i$. Let us label
  these intervals
  \begin{equation*}
    \calE_m^S\cap J_i = I_1\cup I_2 \cup \dots \cup I_{b},
  \end{equation*}
  so that the random variables
  $\lambda(I_k\cap\calE_m^Q\cap \calE_n^Q\cap J_i)$ ($k=1,\dots, b$)
  are independent and
  \begin{equation*}
    \lambda(\calE_m^Q\cap \calE_n^Q\cap J_i) = \sum_{k=1}^{b}\lambda(I_k\cap\calE_m^Q\cap \calE_n^Q\cap J_i).
  \end{equation*}
  Hence, we have 
  \begin{equation*}
    \sigma^2\brackets*{\lambda(\calE_m^Q\cap \calE_n^Q\cap J_i)} = \sum_{k=1}^{b}\sigma^2\brackets*{\lambda(I_k\cap\calE_m^Q\cap \calE_n^Q\cap J_i)}.
  \end{equation*}
  Let $q_k$ ($k=1, \dots, b$) denote the number of intervals
  of $\calE_n^S$ which intersect $I_k$. Each of these intersections has measure
  bounded above by
  $\delta := 2\min\set*{\frac{\psi(m)}{m}, \frac{\psi(n)}{n}}$. Note then that
  \begin{align*}
   \sigma^2\brackets*{\lambda(I_k\cap\calE_m^Q\cap \calE_n^Q\cap J_i)} 
    &\leq \EE\brackets*{\lambda(I_k\cap\calE_m^Q\cap \calE_n^Q\cap J_i)^2}\\
    &\leq \delta^2 p_m\sum_{j=0}^{q_k}\binom{q_k}{j} p_n^j(1-p_n)^{q_k-j} j^2\\
    &= \delta^2 p_mq_k\sum_{j=1}^{q_k}\binom{q_k-1}{j-1} p_n^j(1-p_n)^{q_k-j} j\\
    &\leq \delta^2 p_mp_n q_k^2.
  \end{align*}
  Accordingly,
  \begin{equation}\label{eq:sigma}
    \sigma^2\brackets*{\lambda\parens*{\calE_m^Q\cap\calE_n^Q\cap J_i}} 
    \leq p_mp_n \delta^2 \sum_{k=1}^{b} q_k^2.
  \end{equation}
  Note that we must have $q_k \leq \parens*{\frac{2\psi(m)}{m}} n + 3$
  for all $k=1, \dots, b$, so
  \begin{align*}
    \sum_{k=1}^{b} q_k^2
    &\leq \sum_{k=1}^{b} \parens*{\parens*{\frac{2\psi(m)}{m}} n + 3}^2\\
    &\leq  m\parens*{\parens*{\frac{2\psi(m)}{m}} n + 3}^2.
  \end{align*}
  Putting this into~(\ref{eq:sigma}), a routine inspection reveals
  that 
  \begin{equation*}
    \sigma^2\brackets*{\lambda\parens*{\calE_m^Q\cap\calE_n^Q\cap J_i}} \ll p_mp_n\psi(m)\psi(n).
  \end{equation*}
  Together with~(\ref{eq:twoterms}) this proves the lemma.
\end{proof}

\section{Proof of Theorem~\ref{thm}}
\label{sec:proof}

We have been working toward the following lemma, which states that we
may apply Lemma~\ref{lem:reversebc} in a way described in
Section~\ref{sec:strategy-proof}.

\begin{lemma}[Almost sure quasi-independence on average]\label{lem:asqia}
  Assume we are in the divergence case of Theorem~\ref{thm}. Let
  $M\in \NN$ and express $[0,1] = J_1\cup\dots\cup J_M$, where the
  $J_i$ are  disjoint intervals of length $1/M$. Then, $\PP$-almost surely, for
  any $J\in \set{J_1, \dots, J_M}$ we have
  \begin{equation*}
    \sum_{1\leq m,n \leq N} \lambda\parens*{\calE_m^Q\cap\calE_n^Q\cap J} \ll \frac{1}{\lambda(J)}\parens*{\sum_{n\leq N}\lambda\parens*{\calE_n^Q\cap J}}^2
  \end{equation*}
  for infinitely many $N$.
\end{lemma}

\begin{proof}
  Let us begin by supposing that $J$ is chosen randomly uniformly from
  $\set{J_1, \dots, J_M}$ as in the lemmas of
  Sections~\ref{sec:overlaps} and~\ref{sec:vari-calc}. Combining
  Lemmas~\ref{lem:varone} and~\ref{lem:vartwo}, we have
  \begin{align*}
    \overline\sigma^2\brackets*{\sum_{1\leq m,n\leq N} \lambda\parens*{\calE_m^Q\cap\calE_n^Q\cap J}} 
    &\ll \sum_{1\leq m \neq n \leq N}p_m p_n\psi(m)\psi(n) +\sum_{1 \leq m\neq n \leq N}\EE\brackets*{\lambda\parens*{\calE_m^Q\cap\calE_n^Q}} \\
    &\quad + \sum_{1\leq n \leq N} n p_n(1-p_n)\parens*{\frac{\psi(n)}{n}}^2 \\
    &\ll \parens*{\sum_{1\leq n \leq N}p_n\psi(n)}^2 +\EE\brackets*{\sum_{1 \leq m\neq n \leq N}\lambda\parens*{\calE_m^Q\cap\calE_n^Q}} \\
    &\ll \parens*{\sum_{1\leq n \leq N}p_n\psi(n)}^2,
  \end{align*}
  where the last line is obtained by using~(\ref{eq:inviewof}) and
  Lemma~\ref{lem:expectedoverlaps}. Given that the last expression
  increases to $\infty$ as $N$ increases, we are guaranteed by
  Chebyshev's inequality and the Borel--Cantelli lemma that for any $c>0$ and sufficiently sparse sequence of $N$'s, we
  $\overline\PP$-almost surely have
  \begin{equation*}
    \sum_{1\leq m,n \leq N} \lambda\parens*{\calE_m^Q\cap\calE_n^Q\cap J} \ll \overline\EE\brackets*{\sum_{1\leq m,n \leq N} \lambda\parens*{\calE_m^Q\cap\calE_n^Q\cap J}} + c\parens*{\sum_{1\leq n \leq N}p_n\psi(n)}^2
  \end{equation*}
  for infinitely many $N$ from said sequence. (In particular, we
  should take a sufficiently sparse subsequence of the sequence
  $\set{N_t}$ appearing in Lemma~\ref{lem:asmoas} for the last bound in the next line to hold.) Taking $c\ll \lambda(J)$ the latter is
  bounded by
  \begin{equation*}
    \overset{\textrm{Lem.~\ref{lem:expectedoverlaps}}}{\ll} \lambda(J) \parens*{\sum_{n=1}^N p_n \psi(n)}^2 \overset{\textrm{Lem.~\ref{lem:asmoas}}}{\ll} \frac{1}{\lambda(J)}\parens*{\sum_{n=1}^N \lambda\parens*{\calE_n^Q\cap J}}^2.
  \end{equation*}
  Now, by Fubini's theorem we may conclude that $\PP$-almost surely these
  estimates hold for \emph{every} $J\in \set{J_1, \dots, J_M}$, as required.
\end{proof}

We are now ready to state the proof of Theorem~\ref{thm}.

\begin{proof}[Proof of the divergence part of Theorem~\ref{thm}]
  Assume the conditions of the theorem and furthermore suppose,
  without loss of generality, that $\psi$ takes values in $[0,
  1/4]$.
  Lemma~\ref{lem:asqia} implies that it is almost surely the case that
  for any $M\in\NN$ and any arc $J\subset\set{J_1, \dots, J_M}$ we have
  \begin{equation*}
    \limsup_{N\to\infty}\parens*{\sum_{n=1}^N\lambda\parens*{\calE_n^Q\cap J}}^2 \parens*{\sum_{m,n =1}^N \lambda\parens*{\calE_m^Q\cap\calE_n^Q\cap J}}^{-1}\gg \lambda(J),
  \end{equation*}
  with an implied constant that does not depend on $J$. Since $\sum_{n=1}^\infty\lambda\parens*{\calE_n^Q\cap J}$ diverges by Lemma~\ref{lem:asmoas}, Lemma~\ref{lem:reversebc} implies that almost surely we have 
  \begin{equation*}
    \lambda\parens*{\limsup_{n\to\infty}\calE_n^Q\cap J} \gg \lambda(J) 
  \end{equation*}
  for any $M\in \NN$ and $J\subset\set{J_1, \dots, J_M}$. By a simple
  approximation argument this actually holds for an arbitrary interval
  $J\subset[0,1]$ and thus establishes that almost surely
  $\limsup_{n\to\infty} \calE_n^Q$ has full measure in $[0,1]$  (see, for
  example, \cite[Lemma~1.6]{Harman}, which is a
  standard application of the Lebesgue density theorem). Finally, since
  $\limsup_{n\to\infty} \calE_n^Q\subset W^P(\psi)$, we almost surely
  have $\lambda(W^P(\psi))=1$, as desired.
\end{proof}

\begin{proof}[Proof of the convergence part of Theorem~\ref{thm}]
  Assume that $\sum p_n \psi(n)\neq 0$, for
  otherwise there is nothing to prove. 

  Note that for every $N$
  \begin{equation*}
    \EE\brackets*{\sum_{n=1}^N \lambda\parens*{\calE_n^P}} =2 \sum_{n=1}^N p_n \psi(n). 
  \end{equation*}
  Also, by independence, 
  \begin{align*}
    \sigma^2\brackets*{\sum_{n=1}^N \lambda\parens*{\calE_n^P}} 
    &=\sum_{n=1}^N \sigma^2\brackets*{\lambda\parens*{\calE_n^P}} \\
    &=\sum_{n=1}^N \parens*{\frac{2\psi(n)}{n}}^2p_n (1-p_n) n,
  \end{align*}
  which also converges, say, to $\sigma^2>0$. By Chebyshev's inequality, for any $c>1$, 
  \begin{equation*}
    \PP\brackets*{\sum_{n=1}^N \lambda\parens*{\calE_n^P} \geq c \EE\brackets*{\sum_{n=1}^N \lambda\parens*{\calE_n^P}}} \leq \frac{\sigma^2}{(c-1)^2 \EE\brackets*{\sum_{n=1}^N \lambda\parens*{\calE_n^P}}^2}.
  \end{equation*}
  Choose $c>1$ large enough, so that the right-hand side is $<1/2$ for 
  all large $N$. Now we see that for all large $N$
  \begin{equation*}
    \PP\brackets*{\sum_{n=1}^N \lambda\parens*{\calE_n^P}  <  c \EE\brackets*{\sum_{n=1}^N \lambda\parens*{\calE_n^P}}} \geq \frac{1}{2}. 
  \end{equation*}
  Therefore, we  almost surely have that for infinitely many $N$ 
  \begin{equation}\label{last}
    \sum_{n=1}^N \lambda\parens*{\calE_n^P}  <  c \EE\brackets*{\sum_{n=1}^N \lambda\parens*{\calE_n^P}} = 2c\sum_{n=1}^N p_n \psi(n).
  \end{equation}
  From the facts that the right-most sum in (\ref{last}) converges and the left-most sum
  increases with $N$, we conclude that, almost surely,
  $\sum_{n=1}^\infty \lambda\parens*{\calE_n^P}$ converges. Thus, by the
  Borel--Cantelli lemma, $\lambda\parens*{W^P(\psi)}=0$ and the proof is complete.
\end{proof}

\subsection*{Acknowledgments}
\label{acknowledgments}

The first author was supported by the Austrian Science Fund project
F5510-N26. We started this project at the suggestion of Christoph
Aistleitner, who observed that the partial reduction method
from~\cite{LaimaPaper} could be useful for the problems posed
in~\cite{random_fractions}, and generously communicated the idea to
us.


\bibliographystyle{plain}

\bibliography{random_cognates.bib}

\end{document}